\documentclass{amsart}
\input cyracc.def

\usepackage{amssymb}
\usepackage{amsthm}

\theoremstyle{plain}
\newtheorem{thm}{Theorem}[section]
\newtheorem{lem}[thm]{Lemma}
\newtheorem{cor}[thm]{Corollary}
\newtheorem{prop}[thm]{Proposition}

\theoremstyle{definition}
\newtheorem{dfn}[thm]{Definition}

\def\dnfo{\;\raise.2em\hbox{$\mathrel|\kern-.9em\lower.4em\hbox
{$\smile$}$}}

\def\dnf#1{\lower.9em\hbox{$\buildrel\dnfo\over{ \scriptstyle  #1}$}}

\def\dfo{\;\raise.2em\hbox{$\mathrel|\kern-.9em\lower.4em\hbox{$\smile$}
\kern-.72em\lower.07em\hbox{\char'57}$}\;}
\def\df#1{\lower1em\hbox{$\buildrel\dfo\over{\scriptstyle #1}$}}

\newcommand{\dotminus}{\buildrel\textstyle.\over{\hbox{\vrule height2pt depth0pt width0pt}{\smash-}}}

\newcommand{\minusdot}{\dotminus}

\newcommand{\fA}{\mathfrak{A}}
\newcommand{\fB}{\mathfrak{B}}
\newcommand{\fC}{\mathfrak{C}}

\newcommand{\fL}{\mathfrak{L}}
\newcommand{\fM}{\mathfrak{M}}

\newcommand{\sA}{\mathcal{A}}
\newcommand{\sB}{\mathcal{B}}
\newcommand{\sC}{\mathcal{C}}
\newcommand{\sD}{\mathcal{D}}

\newcommand{\sF}{\mathcal{F}}
\newcommand{\sG}{\mathcal{G}}

\newcommand{\sI}{\mathcal{I}}
\newcommand{\sJ}{\mathcal{J}}

\newcommand{\sL}{\mathcal{L}}
\newcommand{\sM}{\mathcal{M}}

\newcommand{\sR}{\mathcal{R}}
\newcommand{\sS}{\mathcal{S}}

\newcommand{\sX}{\mathcal{X}}

\newcommand{\real}{\mathbb{R}}

\title{Metric Structures and Probabilistic Computation}

\author{Wesley Calvert}
\address{Department of Mathematics \& Statistics\\ Faculty Hall 6C\\
  Murray State University\\ Murray, Kentucky 42071}
\email{wesley.calvert@murraystate.edu}
\date{\today}


\begin{document}

\maketitle

\begin{abstract}
Continuous first-order logic is used to apply
  model-theoretic analysis to analytic structures (e.g.\ Hilbert
  spaces, Banach spaces, probability spaces, etc.).  Classical
  computable model theory is used to examine the algorithmic structure
  of mathematical objects that can be described in classical
  first-order logic.  The present paper shows that probabilistic
  computation (sometimes called randomized computation) can play an
  analogous role for structures described in continuous first-order
  logic.

The main result of this paper is an effective completeness theorem,
showing that every decidable continuous first-order theory has a
probabilistically decidable model.  Later sections give examples of
the application of this framework to various classes of structures,
and to some problems of computational complexity theory.
\end{abstract}

\section{Introduction}

Continuous first-order logic was introduced in \cite{cfo, mtmetric} as
a model-theoretic context sufficient to handle stability theory for
so-called ``metric structures.''  These are many-sorted structures in
which each sort is a complete metric space of finite diameter.  Key
examples include Hilbert spaces, Banach spaces,
probability spaces, and probability spaces with a distinguished automorphism.

For classical first-order model theory, there is a meaningful sense of
computation and effectiveness: In a group, for instance, we have a
reasonable algorithmic understanding of a group if the set of triples
constituting the Caley table (equivalently, the set of words equal to
the identity element) is decidable.  Of course, there are often still
many algorithmic unknowns in the group, such as the conjugacy problem
and the isomorphism problem \cite{dehn1910}.  The aim of the present paper is
to provide a similar framework for continuous first-order logic.

The framework suggested is probabilistic computation.  This model of
computation has seen wide use in complexity theory \cite{impwig,
  sipserbk}, and there is some room for hope that an understanding of
the relationship between continuous and classical first-order logic
might yield insights into the relationship between probabilistic and
deterministic computation.  Section \ref{timecomp} gives reasons for
such hope.

Not to get carried away in speculation, though, it is still cause for
contentment that a way can be found to meaningfully talk about
algorithmic information in the context of metric structures.  The
impossibility of finding an algorithm to solve arbitrary Diophantine
equations (see \cite{h10}), the relationship of isoperimetric functions to
the word problem \cite{isoper1, isoper2}, and much more depend on a notion of
computation adequate to the context of countable rings (in the case of
Diophantine equations) and groups (in the case of the word problem).
Some preliminary results on specific metric structures, given in
Section \ref{examples}, will suggest that there is ground for fruitful
research in the effective theory of these structures.

A key argument that probabilistic computation is the \emph{right}
algorithmic framework for this context is that it admits an effective
completeness theorem.  The classical theorem is this.

\begin{thm}[Effective Completeness Theorem]\label{classeffcomp} A
  (classically) decidable theory has a (classically) decidable
  model.\end{thm}

A full proof of this result may be found in \cite{harizanovhrm}, but
it was known much earlier, at least to Millar \cite{millarfound}.  The
main theoretical contribution of the present paper will be to
interpret the terms of this theorem in such a way as to apply to
a continuous first-order theory and probabilistic computation.  The
main result of the present paper is the proof in Section
\ref{completeness} of the theorem naturally corresponding to Theorem \ref{classeffcomp}.

Section \ref{seccfo} will describe the syntax and semantics for
continuous first-order logic.  The reader familiar with
\cite{cfocomplete} or \cite{mtmetric} will find nothing new in Section
\ref{seccfo}, except a choice of a finite set of logical connectives (no such choice
is yet canonical in continuous first-order logic).  Section
\ref{probcomp} will define probabilistic Turing machines and the
class of structures they compute.  Section \ref{completeness} will contain
the proof of the main result.  Section \ref{examples} will contain some
examples, exhibiting different aspects of the information which is
conveyed by the statement that a certain structure is
probabilistically computable, and in Section \ref{timecomp} we will
conclude with some remarks on time complexity of structures.

\section{Continuous First-Order Logic}\label{seccfo}

We will, in keeping with the existing literature on continuous
first-order logic, adopt the slightly unusual
convention of using $0$ as a numerical value for True (or acceptance)
and $1$ as a numerical value for False (or rejection).  The authors of
\cite{cfo} chose this convention to emphasize the metric nature of
their logic.

Continuous first-order logic is an extension of {\L}ukasiewicz
propositional logic.  The following definitions are from
\cite{cfocomplete}.

\subsection{Semantics}

\begin{dfn} A \emph{continuous signature} is an object of the form
  $\sL = (\sR, \sF, \sG, n)$ where
\begin{enumerate}
\item $\sR$ and $\sF$ are disjoint and $\sR$ is nonempty, and 
\item $n$ is a function associating to each member of $\sR \cup \sF$ its arity
\item $\mathcal{G}$ has the form $\{\delta_{s,i}
  : (0,1] \to (0,1] : s \in \mathcal{R} \cup \mathcal{F} \mbox{ and }
  i < n_s\}$
\end{enumerate}
\end{dfn}

Members of $\sR$ are called \emph{relation symbols}, and members of $\sF$
\emph{function symbols}.  We now define the class of structures.

\begin{dfn} Let $\sL = (\sR, \sF, \sG, n)$ be a continuous signature.
  A \emph{continuous $\sL$-pre-structure} is an ordered pair $\fM = (M,
  \rho)$, where $M$ is a non-empty set, and $\rho$ is a function on
  $\sR \cup \sF$ such that
\begin{enumerate}
\item To each function symbol $f$, the function $\rho$ assigns
  a mapping $f^\fM:M^{n(f)} \to M$
\item To each relation symbol $P $, the function $\rho$ assigns
  a mapping $f^\fM:M^{n(P)} \to [0,1]$.
\item The function $\rho$ assigns $d$ to a pseudo-metric $d^\fM : M
  \times M \to [0,1]$.
\item For each $f \in \mathcal{F}$ for each $i < n_f$, and for each
  $\epsilon \in (0,1]$, we have \[\forall \bar{a}, \bar{b}, c, e
  \left[ d^\fM(c,e) < \delta_{f,i} \Rightarrow d^\fM\left(f^\fM(\bar{a}, c,
  \bar{b}), f^\fM(\bar{a}, e, \bar{b})\right) \leq \epsilon\right]\]
  where $lh(\bar{a}) = i$ and $lh(\bar{a}) + lh(\bar{b}) = n_f-1$.
\item For each $P \in \mathcal{R}$ for each $i < n_P$, and for each
  $\epsilon \in (0,1]$, we have \[\forall \bar{a}, \bar{b}, c, e
  \left[ d^\fM(c,e) < \delta_{f,i} \Rightarrow |P^\fM(\bar{a}, c,
  \bar{b}) - P^\fM(\bar{a}, e, \bar{b})| \leq \epsilon\right]\]
  where $lh(\bar{a}) = i$ and $lh(\bar{a}) + lh(\bar{b}) = n_P-1$.
\end{enumerate}
\end{dfn}

\begin{dfn} A \emph{continuous weak $\sL$-structure} is a continuous
  $\sL$-pre-structure such that $\rho$ assigns to $d$ a
  metric.\end{dfn}

Since we are concerned here with countable structures (i.e.\ those
accessible to computation), we will not use the stronger notion of a
\emph{continuous $\sL$-structure} common in the literature, which
requires that $\rho$ be assigned to a \emph{complete} metric.
However, it is possible, given a continuous weak structure (even a
pre-structure), to pass to a completion \cite{cfocomplete}.

\begin{dfn} Let $V$ denote the set of variables, and let $\sigma: V
  \to M$.  Let $\varphi$ be a formula.
\begin{enumerate}
\item The \emph{interpretation under $\sigma$} of a term $t$ (written
  $t^{\fM,\sigma})$ is defined by replacing each variable $x$ in $t$
  by $\sigma(x)$.
\item Let $\varphi$ be a formula.  We then define the \emph{value of
  $\varphi$ in $\fM$ under $\sigma$} (written $\fM(\varphi, \sigma)$)
  as follows:
\begin{enumerate}
\item $\fM(P(\bar{t}), \sigma) := P^\fM(\overline{t^{\fM, \sigma}})$
\item $\fM(\alpha \dotminus \beta, \sigma) := \max{(\fM(\alpha,
  \sigma) - \fM(\beta, \sigma), 0)}$
\item $\fM(\lnot \alpha, \sigma) := 1- \fM(\alpha, \sigma)$
\item $\fM(\frac{1}{2} \alpha, \sigma) := \frac{1}{2} \fM(\alpha,
  \sigma)$
\item $\fM(\sup_x{\alpha}, \sigma) := \sup\limits_{a \in
  M}{\fM(\alpha, \sigma_x^a)}$, where $\sigma_x^a$ is equal to
  $\sigma$ except that $\sigma_x^a(x) = a$.
\end{enumerate}
\item We write $(\fM, \sigma) \models \varphi$ exactly when
  $\fM(\varphi, \sigma) = 0$.
\end{enumerate}
\end{dfn}
Of course, if $\varphi$ has no free variables, then the value of
$\fM(\varphi, \sigma)$ is independent of $\sigma$.

\subsection{Syntax}

\begin{dfn} Let $\sS_0$ be a set of distinct propositional symbols.
  Let $\sS$ be freely generated from $\sS_0$ by the formal binary
  operation $\dotminus$ and the unary operations $\lnot$ and
  $\frac{1}{2}$.  Then $\sS$ is said to be a \emph{continuous
  propositional logic}.\end{dfn}

We now define truth assignments for continuous propositional logic.

\begin{dfn} Let $\sS$ be a continuous propositional logic.
\begin{enumerate}
\item if $v_0: \sS_0 \to [0,1]$ is a mapping, we can extend $v_0$ to a
  unique mapping $v: \sS \to [0,1]$ by setting
\begin{enumerate}
\item $v(\varphi \dotminus \psi) := \max{(v(\varphi) - v(\psi),0)}$
\item $v(\lnot\varphi) := 1-v(\varphi)$
\item $v(\frac{1}{2} \varphi) = \frac{1}{2} v(\varphi)$
\end{enumerate}
We say that $v$ is the \emph{truth assignment} defined by $v_0$.
\item We write $v \models \Sigma$ for some $\Sigma \subseteq \sS$
  whenever $v(\varphi) = 0$ for all $\varphi \in \Sigma$.
\end{enumerate}
\end{dfn}

Roughly, $\varphi \dotminus \psi$ has the sense of $\psi \to
\varphi$.  We can also, of course, define $\varphi \wedge \psi$ as
$\varphi \dotminus (\varphi \dotminus \psi)$, and $\varphi \vee \psi$
via deMorgan's law.  We can also define something resembling
equivalence, $|\varphi - \psi| = (\varphi \dotminus \psi) \vee (\psi
\dotminus \varphi)$.  {\L}ukasiewicz propositional logic is the
fragment of this logic which does not involve $\frac{1}{2}$.

To make a first-order predicate variant of this logic, we use $\sup$
in the place of $\forall$ and $\inf$ in the place of $\exists$ (with
the obvious semantics, as will be described in what follows).  We
typically also include a binary function $d$, whose standard
interpretation is generally a metric.  Now we
give the syntactic axioms for continuous first-order logic:

\newcounter{axiom}
\begin{list}{(A\arabic{axiom})}{\usecounter{axiom}}
\item $(\varphi \dotminus \psi) \dotminus \varphi$
\item $((\chi \dotminus \varphi) \dotminus (\chi \dotminus \psi))
  \dotminus (\psi \dotminus \varphi)$
\item $(\varphi \dotminus (\varphi \dotminus \psi)) \dotminus (\psi
  \dotminus (\psi \dotminus \varphi))$
\item $(\varphi \dotminus \psi) \dotminus (\lnot \psi \dotminus \lnot
  \varphi)$
\item $\frac{1}{2} \varphi \dotminus (\varphi \dotminus \frac{1}{2}
  \varphi)$
\item $(\varphi \dotminus \frac{1}{2} \varphi) \dotminus \frac{1}{2}
  \varphi$.
\item $(\sup_x{\psi} \dotminus \sup_x{\varphi}) \dotminus \sup_x{(\psi
  \dotminus \varphi)}$
\item $\varphi[t/x] \dotminus \sup_x{\varphi}$ where no variable in
  $t$ is bound by a quantifier in $\varphi$.
\item $\sup_x{\varphi} \dotminus \varphi$, wherever $x$ is not free in
  $\varphi$.
\item $d(x,x)$
\item $d(x,y) \minusdot d(y,x)$
\item $\left(d(x,z) \minusdot d(x,y)\right) \minusdot d(y,z)$
\item For each $f \in \sF$, each $\epsilon \in (0,1]$, and each $r, q
  \in \sD$ with $r > \epsilon$ and $q < \delta_{f,i}(\epsilon)$, the
  axiom $\left(q \minusdot d(z,w)\right) \wedge \left(d\left(f(\bar{x}, z,
  \bar{y}), f(\bar{x}, w, \bar{y})\right) \minusdot r\right)$, where
  $lh(\bar{x}) + lh(\bar{y}) = n_f-1$.
\item For each $P \in \sR$, each $\epsilon \in (0,1]$, and each $r, q
  \in \sD$ with $r > \epsilon$ and $q < \delta_{P,i}(\epsilon)$, the
  axiom $\left(q \minusdot d(z,w)\right) \wedge \left(\left(P(\bar{x}, z,
  \bar{y}) \minusdot P(\bar{x}, w, \bar{y})\right) \minusdot r\right)$, where
  $lh(\bar{x}) + lh(\bar{y}) = n_P-1$.
\end{list}

Axioms A1--A4 are those of {\L}ukasiewicz propositional logic, and
axioms A5--A6 are those of continuous propositional logic.  Axioms
A7--A9 describe the role of the quantifiers.  Axioms A10--A12
guarantee that $d$ is a pseudometric, and axioms A13--A14 guarantee
uniform continuity of functions and relations.  We write $\Gamma
\vdash_Q \varphi$ whenever $\varphi$ is provable from $\Gamma$ in
continuous first-order logic.  Where no confusion is likely, we will
write $\Gamma \vdash \varphi$.

\section{Probabilisticly Computable Structures}\label{probcomp}

If $M$ is a Turing machine, we write $M^x(n)$ for the result of
applying $M$ to input $n$ with oracle $x$.  Excepting the polarity
change to match the conventions above, the following definition
is standard; it may be found, for instance, in \cite{sipserbk}.

\begin{dfn} Let $2^\omega$ be the set of infinite binary sequences,
  with the usual Lebesgue probability measure $\mu$.
\begin{enumerate}
\item A \emph{probabilistic Turing machine} is a Turing machine
  equipped with an oracle for an element of $2^\omega$, called the
  \emph{random bits}, with output in $\{0,1\}$.
\item We say that a probabilistic Turing machine $M$ \emph{accepts $n$
  with probability $p$} if and only if $\mu \{x \in 2^\omega :
  M^x(n)\downarrow = 0\} = p$.
\item We say that a probabilistic Turing machine $M$ \emph{rejects $n$
  with probability $p$} if and only if $\mu \{x \in 2^\omega :
  M^x(n)\downarrow = 1\} = p$.
\end{enumerate}
\end{dfn}

\begin{dfn} Let $\sL$ be a computable continuous signature.  Let $\fM$
  be a continuous $\sL$-structure.  Let $\sL(\fM)$ be the expansion of
  $\sL$ by a constant $c_m$ for each $m \in M$ (i.e.\ a unary
  predicate $c_m \in \sR$ where $c_m^\fM(x) := d(x,m)$).  Then the
  \emph{continuous atomic 
  diagram} of $\fM$, written $D(\fM)$ is the set of all pairs
$(\varphi, p)$, where $\varphi$ is a quantifier-free (i.e.\ $\sup$-
and $\inf$-free) sentence in $\sL(\fM)$ and $\fM(\varphi, \sigma) =
p$.  The continuous elementary diagram $D^*(\fM)$ is the same, except that
  $\varphi$ is not required to be quantifier-free.\end{dfn}

Note that the definition is independent of $\sigma$, since a
sentence has no free variables.

\begin{dfn} We say that a continuous pre-structure $\fM$ is
  \emph{probabilistically computable} (respectively,
  \emph{probabilistically decidable}) if and only if there is some
  probabilistic Turing machine $T$ such that, for every pair
  $(\varphi, p) \in D(\fM)$ (respectively, $D^*(\fM)$) the machine $T$ accepts $\varphi$ with
  probability $p$. \end{dfn}

Suppose $T$ is a deterministic machine (i.e.\ one that makes no use of
its random bits; a classical Turing machine) and $\fM$ a classical
first-order structure.  Then this definition corresponds exactly to
the classical definition of a computable structure.

We cannot do entirely without the probabilistic machines (that is, we
cannot thoroughly understand probabilistically computable structures
using only classical Turing machines), as the follwoing result shows.

\begin{lem}[No Derandomization Lemma] There is a probabilistically
  computable weak structure $\fM$ such that the set $\{(\varphi, p)
  \in D(\fM) : p \in \sD \}$ is not classically computable.\end{lem}

\begin{proof} Let $U$ be a computably enumerable set, and let $S$ be
  the complement of $U$.  We first construct a probabilistically
  computable function $f$ such that \[P(f^\sigma(x) = 0) = \frac{1}{2}\]
  if and only if $x \in S$.  At stage $t$, if $x \in U_t$, pick two
  strings $\sigma_t, \tau_t$ of length $t+2$ such that $f_t(x)$ does
  not halt with random bits $\sigma_t$ or $\tau_t$.  We define the
  function $f_{t+1} := f_t \cup \{f^{\sigma_t}(x) = 0, f^{\tau_t}(x) =
  1\}$.  On the other hand, if $x \in U_t$, then we arrange that
  $f_{t+1}^\sigma(x) = 0$ for all $\sigma$ of length at most $t+2$
  where $f_t^\sigma(x)$ does not halt.  Let $f = \bigcup\limits_{t \in
  \omega} f_t$.  Now if $x \in S$, we never see
  $x \in U_t$, so $f(x) = 0$ with probability $\frac{1}{2}$.
  Otherwise, there is some $t$ such that $x \in U_t - U_{t-1}$, and
  then $f(x) = 0$ with probability
  $1 - \sum\limits_{i = 2}^t 2^{-i} > \frac{1}{2}$.

Now we let $\fM$ be the structure $(\omega, f)$, where $f$ is
interpreted as a unary predicate in the obvious way, and $d$ is the
discrete metric.  If we could
decide membership in $\{(\varphi, p) \in D(\fM) : p \in \sD \}$ with a
classical Turing machine, then we could also decide membership in
$U$.\end{proof}

Of course, this same argument could work for any other uniformly
computable set of reals in place of $\sD$.  To some extent, though, we
could do without the probabilistic machines.  The following results
show that the statement of the No Derandomization Lemma is the
strongest possible.

\begin{prop}\label{sfunction} For any probabilistically computable pre-structure $\fM$,
  there is some (classically) computable function $f$, monotonically
  increasing in the second variable, and some (classically) computable
  function $g$, monotonically decreasing in the second variable, such
  that for any pair $(\varphi, p) \in D(\fM)$, we have $\lim\limits_{s
  \to \infty} f(\varphi, s) = p$ and $\lim\limits_{s \to \infty}
  g(\varphi, s) = p$.\end{prop}

\begin{proof} Let $\fM$ be computed by the probabilistic Turing
  machine $T_\fM$.  Let $(\sigma_s)_{s \in \omega}$ be an effective
  list of all strings in $2^{< \omega}$.  Now we define \[f(\varphi,
  s) := \sum\limits_{i \leq s} \left(T_\fM^{\sigma_i}(\varphi) \cdot
  2^{-lh(\sigma_i)}\right).\]  The definition of $g$ is symmetric.
  These clearly have the correct properties.\end{proof}

Functions of the same form as $f$ and $g$ are often seen in classical
computable model theory \cite{khisamievhrm, ordercomp, cchmeq,
  cchmapg}.

\begin{cor} For any probabilistically computable pre-structure $\fM$,
  the set of pairs, $(\varphi, p) \in D(\fM)$ is the complement of a
  (classically) computably enumerable set.\end{cor}

\begin{proof} Follows immediately from Proposition
  \ref{sfunction}. \end{proof}

These limitations notwithstanding, the definition via probabilistic
machines gives a more natural continuity with the established
literature on continuous first-order model theory \cite{mtmetric}.  In
addition, this definition is in any case not dispensable when, for
instance, time complexity of computation is at issue (see Section
\ref{timecomp}).

\section{Effective Completeness}\label{completeness}

Theorem \ref{classeffcomp} is an important piece of evidence that
classical Turing computation (or any of the many equivalent concepts)
is properly synchronized with classical first-order logic.  In
particular, it asserts that under the minimal, obviously necessary
hypotheses, a classical first-order theory has a model which can be
represented by a classical computation.  The aim of the present
section is to prove a similar result for continuous first-order logic
and probabilistic computation.  The following analogue to the
classical concept of the decidability of a theory was proposed in
\cite{cfocomplete}.

\begin{dfn} Let $\sL$ be a continuous signature and $\Gamma$ a set of
  formulas of $\sL$.  
\begin{enumerate}
\item We define \[\varphi_\Gamma^\circ :=
  \sup{\{\fM(\varphi, \sigma) : (\fM, \sigma) \models \Gamma\}}.\]
\item If $T$ is a complete continuous first-order theory, we say
  that $T$ is \emph{decidable} if and only if there is a (classically)
  computable function $f$ such that $f(\varphi)$ is an index for a
  computable real number equal to $\varphi_T^\circ$.
\end{enumerate}
\end{dfn}

\begin{thm}\label{effcomplwk} Let $T$ be a decidable continuous
  first-order theory.  Then there is a probabilistically decidable
  continuous pre-structure $\fM$ such that $\fM \models T$.\end{thm}

\begin{proof}  The construction of a model $\fM$ is given in
  \cite{cfocomplete}, by an analogue of Henkin's method.  Our model
  will be essentially the same, except that some care must be taken
  with effectiveness.  The principal content of the theorem consists
  in showing that this structure is probabilistically decidable.  We
  will define a probabilistic Turing machine which, for any formula
  $\varphi$, accepts $\varphi$ with probability $\fM(\varphi)$.

We begin by adding Henkin witnesses.  Let $\sD$ denote the dyadic
numbers in the interval (i.e.\ those of the form $\frac{k}{2^n}$ for
$k, n \in \mathbb{N}$).

\begin{dfn} Let $\Gamma$ be a set of formulae.  Then
  $\Gamma$ is said to be \emph{Henkin complete} if for every formula
  $\varphi$, every variable $x$, and every $p < q \in \sD$, there is a
  constant $c$ such that \[(\sup_x \varphi \dotminus q) \wedge (p
  \dotminus \varphi [c/x]) \in \Gamma.\]\end{dfn}

\begin{lem} We can effectively extend $T$ to a consistent set $\Gamma$
  of formulae which is Henkin complete. \end{lem}

\begin{proof} Let $\sL_0 = \sL$.  For each $n$, let $\sL_{n+1}$ be
the result of adding, for each formula $\varphi$ in $\sL_n$, and for each $x,
p, q$ as in the previous definition, a new constant $c_{(\varphi, x, p,
  q)}$.  We can also extend the theory $T$, beginning with $\Gamma_0 = T$.
For each $n$, the set $\Gamma_{n+1}$ is produced by adding to $T$, for
each formula $\varphi$ in $\sL_n$ and each $x, p, q$ as in the
previous definition, the formula $(\sup_x \varphi \dotminus q) \wedge
(p \dotminus \varphi[c_{(\varphi, x, p, q)}/x])$.  Let $\Gamma =
\bigcup\limits_n \Gamma_n$.  The consistency of $\Gamma$ is
demonstrated in \cite{cfocomplete}.

Note that this construction is in every way effective.  In particular,
there is a (classically) computable function which will, given $p, q
\in \sD$ and G\"odel numbers for $\varphi$ and $x$, give us a G\"odel
number for $c_{(\varphi, x, p, q)}$.  Moreover, the set $\Gamma$ is
(classically) computable. \end{proof}

We write $\sL^* = \bigcup\limits_n \sL^n$, and $C = \{c_{(\varphi, x,
  p, q)}\}$.

\begin{lem} We can effectively extend $\Gamma$ to a consistent set
  $\Delta^0$ such that for all formulae $\varphi, \psi$ in $\sL^*$ we
  have $\varphi \dotminus \psi \in \Delta$ or $\psi \dotminus \varphi
  \in \Delta^0$.\end{lem}

\begin{proof} We set $\Delta_0 = \Gamma$.  At stage $s+1$, for each pair $\psi, \varphi$ of
  sentences from $\sL^*$ such that neither $\psi \dotminus \varphi$
  nor $\varphi \dotminus \psi$ is in $\Delta_{s}$, we proceed as follows.  Let $\theta$ be
  the conjunction of all elements of $\Delta_{s}$, and let
  $\bar{c}$ be the constants from $C$ which occur in $(\psi \dotminus
  \varphi) \dotminus \theta$.  We then check (effectively, since the
  theory is decidable), whether $\left(\forall \bar{x} \left((\psi
  \dotminus \varphi) \dotminus \theta\right) (\bar{x}/\bar{c})
  \right)^\circ_T = 0$.  If so, then we add $\psi \dotminus \varphi$
  to form $\Delta_{s+1}$.  Otherwise, we do so with $\varphi
  \dotminus \psi$.  Now $\Delta^0 = \bigcup\limits_s \Delta_s$ is as
  required.  That this extension is consistent is established in
  \cite{cfocomplete}.\end{proof}

\begin{dfn} Let $\Delta$ be a set of formulas.  We say that $\Delta$
  is \emph{maximal consistent} if $\Delta$ is consistent and for all
  formulae $\varphi, \psi$ we have 
\begin{enumerate}
\item If $\Delta \vdash \varphi \dotminus 2^{-n}$ for all $n$, then
  $\varphi \in \Delta$, and
\item $\varphi \dotminus \psi \in \Delta$ or $\psi \dotminus \varphi
  \in \Delta$.
\end{enumerate}
\end{dfn}

Now let $\Delta^0 = \bigcup\limits_s \Delta_s$, and \[\Lambda =
\{\varphi : \forall n [ \delta^0 \vdash \varphi \dotminus 2^{-n}]\}.\]
Now $\Delta = \Delta^0 \cup \Lambda$ is maximal consistent, by
construction of $\Delta^0$.  Let $\fM$ be the model of $T$ whose
universe is the set of closed terms in $C$, as in \cite{cfocomplete}.

We now define the probabilistic Turing machine $G$ which will witness
that $\fM$ is probabilistically computable.  We set $K^A_{0} =
K^R_{0} = A_0 = R_0 = \emptyset$.  We define the functions $E(S) =
\{\sigma \supseteq \tau : \tau \in S\}$ and  \[P(S) =
\sum\limits_{\sigma \in S} \frac{1}{2^{lh(a)}}.\]

At stage $s$, if $\Delta_s \vdash \varphi \dotminus \frac{k}{2^n}$,
then we will arrange that $G$ accepts $\varphi$ with probability at least
$1-\frac{k}{2^n}$.  If $K^A_{s} = \emptyset$, then we find $2^n - k$
nodes $\sigma_1, \dots, \sigma_{2^n - k}$ of length $n$ in
$2^{<\omega} - E(K^R_s)$, and let $K^A_{s+1} = \{\sigma_1, \dots,
\sigma_{2^n - k}\}$.  If $K^A_s$ is nonempty and $P(K^A_s) \geq
1-\frac{k}{2^n}$, then we do nothing with $K^A$.  If $K^A_s$ is
nonempty and $P(K^A_s) < 1-\frac{k}{2^n}$, then we find some set
$\Sigma$ of elements of $2^{<\omega} - E(K^R_s)$ with length $n$ so that $P(K^A_s
\cup \Sigma) = 1 - \frac{k}{2^n}$, and let $K^A_{s+1} = K^A_s \cup \Sigma$.

If $\Delta_s \vdash \frac{k}{2^n} \dotminus \varphi$ then we will
arrange that $G$ rejects $\varphi$ with probability at least
$\frac{k}{2^n}$.  If $K^R_{s} = \emptyset$, then we find $k$ nodes
$\sigma_1, \dots, \sigma_{k}$ of length $n$ in $2^{<\omega} -
E(K^A_s)$, and let $K^R_{s+1} = \{\sigma_1, \dots, \sigma_{k}\}$.  If
$K^R_s$ is nonempty and $P(K^R_s) \geq \frac{k}{2^n}$, then we do
nothing with $K^R$.  If $K^R_s$ is nonempty and $P(K^R_s) <
\frac{k}{2^n}$, then we find some set $\Sigma$ of elements of
$2^{<\omega} - E(K^A_s)$ with length $n$ so that $P(K^R_s \cup \Sigma) =
\frac{k}{2^n}$, and let $K^R_{s+1} = K^R_s \cup \Sigma$.

At this point, it is necessary to verify that certain aspects of the
construction described so far are actually possible.  In particular,
we need to show that when we search for elements of $2^{<\omega} -
E(K^A_s)$, for instance, there will be some.  Now if $E(K^A_s)$
contains more than $2^n - k_1$ elements, we must have $P(K^A_s) < 1 -
\frac{k_1}{2^n}$, so that we must have had $\Delta_s \vdash \varphi
\dotminus \frac{k_1}{2^n}$ (Note that if $\Delta_s \vdash \varphi
\dotminus p$ and $q > p$, then also $\Delta_s \vdash \varphi \dotminus
q$).

\begin{lem} If there is some $s$ such that $\Delta_s \vdash \varphi
  \dotminus \frac{k_1}{2^n}$ and $\Delta_s \vdash \frac{k}{2^n}
  \dotminus \varphi$, then $(1 - \frac{k_1}{2^n}) + \frac{k}{2^n} \leq
  1$.
\end{lem}

\begin{proof} Suppose not.  Then $2^n - k_1 + k > 1$, so that $k - k_1
  > 0$ and $k > k_1$.  However, we also have $\frac{k}{2^n} \dotminus
    \frac{k_1}{2^n} = 0$, so that $k_1 \geq k$, a contradiction.
\end{proof}

The situation for finding elements of $2^{<\omega} - E(K^R_s)$ is
symmetric.

Returning to the construction, at stage $s$, we will add more
instructions.  We will guarantee that for any $\sigma \in E(K^A_s)$,
we will have $G^\sigma(\varphi) \downarrow = 0$, and for any $\sigma
\in E(K^R_s)$ we will have $G^\sigma(\varphi) \downarrow = 1$.

Let $\varphi$ be a sentence in $\sL^*$, and suppose $\fM(\varphi) =
p$.  We need to show that $G$ accepts $\varphi$ with probability $p$.
Since $\Delta$ is maximal consistent, for each $q_0, q_1 \in \sD$ with
$q_0 \leq \fM(\varphi) \leq q_1$, there was some $s$ for which
$\Delta_s \vdash \varphi \dotminus q_1$ and for which $\Delta_s \vdash
q_0 \dotminus \varphi$, and at that stage, we ensured that $G$ would
accept $\varphi$ with probability between $q_0$ and $q_1$.  Since this
is true for all $q_0 \leq p \leq q_1 \in \sD$, it must follow that $G$
accepts $\varphi$ with probability $p$.

%
%
%

\end{proof}

We can strengthen Theorem \ref{effcomplwk} to produce a continuous
weak structure if we require $T$ to be \emph{complete}.

\begin{dfn} Let $\fM$ be a continuous $\sL$-structure.  We write $Th(\fM)$
  for the set of continuous $\sL$-sentences $\varphi$ such that
  $\fM(\varphi) = 0$.  We say that $T$ is \emph{complete} if $T =
  Th(\fM)$ for some $\fM$.\end{dfn}

\begin{cor}\label{effcomplstr} Let $T$ be a complete decidable continuous
  first-order theory.  Then there is a probabilistically decidable
  weak structure $\fM$ such that $\fM \models T$.\end{cor}

\begin{proof} If the signature has no metric, then Theorem
  \ref{effcomplwk} suffices.  Otherwise, we note that $T$ must contain
  the sentence $\sup\limits_{x,y}{\left((x=y) \dotminus
  d(x,y)\right)}$, so that when we apply Theorem \ref{effcomplwk}, the
  function $d^\fM$ is a metric on $\fM$.\end{proof}

\section{Examples}\label{examples}

A full treatment of each of the following classes of examples suggests
a paper --- or many papers --- of its own.  However, in each case some
suggestion is given of the kind of data given by the assumption that
an element of the class is probabilistically computable.

\subsection{Hilbert Spaces}\label{hilbertspaces}

A pre-Hilbert space over a topological field $F$ is a vector space
with an inner product meeting all requirements of being a Hilbert
space except perhaps that it may not be complete with respect to the
norm.  The authors of \cite{mtmetric} identify a pre-Hilbert space $H$
with the many-sorted weak structure \[\sM(H) =
\left(\left(B_n(H) : n \geq 1 \right), 0, \{I_{mn}\}_{m < n},
  \{\lambda_r\}_{r \in F}, +, -, \langle \cdot, \cdot \rangle \right)\]
  where $B_n(H)$ is the closed ball of radius $n$, the map $I_{mn}$ is
  inclusion of $B_m$ in $B_n$, each $\lambda_r$ is a scaling function,
  $+$ and $-$ are the standard vector operations, and $\langle \cdot ,
  \cdot \rangle$ is the inner product.  We also write $||x||$ for
  $\sqrt{\langle x,x \rangle}$, and the structure has a metric given by
  $d(x,y) = ||x-y||$.  Of course, the normal case is to let $F =
  \real$, but this is not necessary.  Now the pre-Hilbert space $H$ is
  clearly a continuous weak structure in the obvious signature.  A
  true Hilbert space is a continuous structure.

\begin{thm}[\cite{mtmetric}] There is a continuous first order theory,
  IHS, such that the following hold:
\begin{enumerate}
\item IHS is (classically) computably axiomatized
\item IHS is complete
\item Any two models of IHS with the same infinite cardinality are
  isomorphic.
\item The continuous first-order structures $\sM$ that satisfy IHS are
  exactly the infinite-dimensional Hilbert spaces.
\item IHS admits quantifier elimination.
\item IHS is $\omega$-stable.
\end{enumerate}
\end{thm}

Since IHS is computably axiomatizable and complete, it must also be
decidable.  Consequently, Theorem \ref{effcomplstr} shows that there
must exist some probabilistically computable weak structure satisfying
IHS.  Clearly one challenge of handling Hilbert spaces from a
computational viewpoint is the essential uncountability of a true
Hilbert space. However, computational scientists are generally
undaunted by this feature, being satisfied with approximations in
place of true limits of Cauchy sequences.  We will adopt a similar
approach.  The following result is well known (see \cite{liebloss} for
a proof).

\begin{lem} For any $p$, the space $L^p(\real)$ is separable (i.e.\
  has a countable dense subset).\end{lem}

Let $F$ be a countable topological field dense in $\real$, let $H$
be a separable Hilbert space over the reals, and let $D$ be a
countable dense subspace of $H$.  We will write $H^F$ for the
restriction of $D$ to the language which includes scalars only from
$F$.  It is well known that $H$ must have a countable orthonormal
basis (see, for instance, \cite{reedsimon1}).

\begin{prop} If $H^F$ is probabilistically computable, then there is a
  probabilistic Turing machine which will, given a probabilistically
  computable countable basis for $H^F$, produce a probabilistically
  computable orthonormal basis for $H^F$.\end{prop}

\begin{proof} The proof is precisely an implementation of the
  Gram--Schmidt process.  Let $(u_i)_{i \in \omega}$ be a
  probabilistically computable countable basis.  Set $v_1 :=u_1$.  For
  $s \geq 2$, at stage $s$, we set \[v_s := u_s - \sum\limits_{i =
  1}^{s-1} \frac{\langle u_s, v_i\rangle}{||v_i||}.\]  Now $(v_i)_{i
  \in \omega}$ is an orthogonal basis for $H^F$.  We can normalize to
  an orthonormal basis $(\tilde{v}_i)_{i \in \omega}$ by setting
  $\tilde{v}_i : = \frac{v_i}{||v_i||}$. \end{proof}

\subsection{Banach Spaces and Banach Lattices}

Certainly, the framework of Section \ref{hilbertspaces} is sufficient,
without the inner product, to account for any normed vector space as a
continuous pre-structure (and thus any Banach space as a continuous
structure).

One also sometimes sees some Banach spaces with the
additional structure of a lattice included.  In such a structure, $f
\vee g$ is the pointwise minimum of $f$ and $g$, and $f \wedge g$ is
the pointwise maximum.  Such an approach is taken in \cite{schaefer,
  lindenstrauss} and also in \cite{mtmetric}. 

Fixed point theorems for operators on Banach spaces or appropriate
subsets of them are an important technique for many non-linear
differential equations \cite{evanspde, gilbargtrudinger, saatybram,
  reedsimon1}.  It is well-known that the classical Brouwer fixed point
theorem is not effectively true (that is, roughly, if the convex hull
$C$ of a nonempty finite set of points in $\real^n$ and a continuous
function $f: C \to C$ are given in a computable way, there may still
be no algorithm to find a fixed point for $f$; see
\cite{simpsonsosa}).  Since the countable spaces we are working with
are not complete, we cannot hope for an exact version of the Banach
fixed point theorem.  However, the following approximate version does
hold (the classical proof is described in \cite{evanspde}; the proof
below is a slight modification).

\begin{prop}[Effective Banach Fixed Point Theorem] There is an
  effective procedure which will, given $\epsilon > 0$, a
  probabilistically computable normed vector space $X$ and a
  computable nonlinear mapping $A: X \to X$ such that \[d\left(A(u) -
  A(v)\right) \leq \gamma d(u,v)\] for some $\gamma < 1$, produce an
  element $x^*$ such that $d\left(x^*, A(x^*)\right) <
  \epsilon$.\end{prop}

\begin{proof} Fix some $u_0 \in X$, and let $i >
  \log_\gamma{\frac{\epsilon}{d\left(A(u_0), u_0\right)}}$.  Such an
  $i$ can be found effectively by standard approximations.  Write
  $u_k$ for $A^k(u_0)$.  Then \[d\left(A(u_{i+1}), u_{i+1}\right) =
  d\left(A(u_{i+1}), A(u_i)\right) \leq \gamma d\left(u_{i+1},
  u_i\right) \leq \cdots \leq \gamma^i d\left(A(u_0), u_0\right) <
  \epsilon.\]  Let $u^* = u_{i+1}$, and the result holds.\end{proof}

This shows that approximate weak solutions to certain differential
equations (for instance, some of reaction-diffusion type, see
\cite{evanspde}) can be found effectively in probabilistically
computable structures.  The key insight of this fixed point result is
this.  The traditional view of effective model theory has asked
whether a particular theorem is ``effectively true.''  By contrast,
the important question for applications is more typically whether the
theorem is ``effectively \emph{nearly} true.''  That is,
approximations are good enough, and often all that is necessary.  For
solutions of differential equations modeling applications, for
instance, practitioners often ``would rather have an accurate
numerical solution of the correct model than an explicit solution of
the wrong model.  Explicit solutions are so rare that fast accurate
numerical analysis is essential'' \cite{optionpricing}.  While it is
not new to observe that classically unsolvable problems can be
effectively approximated, the framework of probabilistically
computable structures in continuous first-order logic is one that
calls attention to the possibility of approximate solutions, rather
than to the impossibility of exact ones.

\subsection{Probability Spaces}

Let $\sX = (X, \sB, \mu)$ be a probability space.  We say that $B \in \sB$
is an \emph{atom} if $\mu(B) > 0$ and there is no $B' \in \sB$ with
$B' \subseteq B$ and $0 < \mu(B') < \mu(B)$.  We say that $\sX$ is
atomless if and only if $\sB$ contains no atoms.  Let $\hat{\sB}$ be
the quotient of $\sB$ by the relation $B_1 \sim B_2$ if and only if
$\mu (B_1 \triangle B_2) = 0$.  The authors of \cite{mtmetric} identified
$\sX$ with the structure \[\left(\hat{\sB}, 0, 1, \cdot^c, \cap, \cup,
\mu\right)\] with the metric $d(A, B) = \mu (A \triangle B)$.

\begin{thm}[\cite{mtmetric}] There is a continuous first-order theory APA
  such that
\begin{enumerate}
\item APA is finitely axiomatizable.
\item APA is complete.
\item APA admits quantifier elimination.
\item The continuous pre-structures satisfying APA are exactly the
  atomless probability spaces, represented as above.
\end{enumerate}
\end{thm}

Since APA is finitely axiomatizable and complete, it is also
decidable.  Thus, Theorem \ref{effcomplstr} gives us a
probabilistically decidable model.  Of course, any separable
probability structure can be approximated by a countable dense set.

A standard issue in effective model theory is whether two isomorphic
structures must be isomorphic via a computable function.  The
following result shows that the answer for probability structures is
affirmative.

\begin{prop} Let $\fB$ and $\fC$ be isomorphic, probabilistically
  computable atomless probability structures with universes $\hat{\sB}$ and
  $\hat{\sC}$, respectively.  Then there is a (classically) computable
  function witnessing the isomorphism.\end{prop}

\begin{proof} The isomorphism is constructed by a standard
  back-and-forth argument.  Suppose that $f: \fB \to \fC$ is a finite
  partial isomorphism, and that $x \in \sB - dom(f)$.  We wish to find
  some $y \in \sC$ such that $f \cup \{(x, y)\}$ is still a partial
  isomorphism.  Without loss of generality, we may assume that $x$ is
  not in the substructure of $\fB$ generated by $dom(f)$.  Let $a_0,
  \dots, a_n$ be the atoms of the substructure of $\fB$ generated by
  $dom(f)$.  We may assume, without loss of generality, that each
  $a_i$ is in $dom(f)$.  Now the isomorphism type of $x$ is determined
  by the values $\mu(x \cap a_i)$.  Since $\sB$ is atomless, there is
  an element $y \in \hat{\sC}$ such that for each $i$ we have
  $\mu^\fB(x \cap a_i) = \mu^\fC(y \cap f(a_i))$, and since $\fB$ and
  $\fC$ are probabilistically computable, we can effectively find this
  $y$.  The extension to a new element of $\fC$ is entirely symmetric.
  The union of the partial isomorphisms constructed in this way will
  be a computable function, and will be an isomorphism from $\fB$ to
  $\fC$.\end{proof}

\subsection{Probability Spaces with a Distinguished Automorphism}

A standard sort of enrichment in stability theory is to expand a known
structure by adding a new function symbol to define a new function,
and to specify that this function be generic.  Fix an interval $I$
under the Lebesgue measure $\lambda$, and let $\fL$ be the algebra of
measurable sets.  Let $G$ denote the group of measure preserving
automorphisms of $(I, \fL, \lambda)$, modulo the relation of
almost everywhere agreement.  Let $\tau \in G$.  Now $\tau$ induces an
automorphism on $(\hat{L}, 0, 1, \cdot^c, \cap, \cup, \lambda)$ in a
straightforward way (see \cite{probwaut, mtmetric}).  We say that
$\tau \in G$ is \emph{aperiodic} if for every positive integer $n$ we
have \[\lambda \{x \in I : \tau^n(x) = x\} = 0.\]  To have a countable
structure of this type, we could take a countable dense subset $I'
\subseteq I$, and for $X \subseteq I'$, set $\lambda(X) = \lambda(cl(X))$.

In \cite{probwaut} and \cite{mtmetric}, an axiomatization is given for
the theory of atomless probability spaces with a distinguished
aperiodic automorphism.  This theory is complete and admits
elimination of quantifiers.  The authors of \cite{probwaut} show that
entropy arises as a model-theoretic rank.

The result below partially describes the degree of algorithmic control
we can expect on iterations of such an automorphism.  Before stating
the result, though, a probabilistic analogue to computable
enumerability should be given:

\begin{dfn} We say that a set is probabilistically computably
  enumerable if and only if there is some probabilistic Turing machine
  $M$ such that
\begin{itemize}
\item If $x \in S$, then for any $q<1$ the machine $M$ accepts
  $x$ with probability at least $q$, and
\item If $x \notin S$, then there is some $q<1$ such that $M$
  accepts $x$ with probability at most $q$.
\end{itemize}
\end{dfn}

In particular (and especially in light of the time complexity
considerations in Section \ref{timecomp}), if we specify an error
tolerance $q$, there is some $s$ such that $M(s,x)$ is below $q$
whenever $x \in S$, and (assuming the tolerance is sufficiently small)
no such $s$ otherwise.

\begin{thm} Let $\sI = (\hat{L}, 0, 1, \cdot^c, \cap, \cup, \lambda, \tau)$
  be a probabilistically computable probability structure based on a
  dense subset of the unit interval, with a measure-preserving
  transformation $\tau$.  Let $A \subseteq I$ be a set of
  positive measure, defined without quantifiers in continuous
  first-order logic.  Write $\fA$ for the set $\bigcup\limits_{n \in
  \omega} \tau^n(A)$.  Then for any isomorphism $f: \sI \to \sJ$ to a
  probabilistically computable structure $\sJ$, the set $f(\fA)$ is
  probabilistically computably enumerable.
\end{thm}

\begin{proof} Toward part 1, note that $\fA$ is defined by the
  infinitary disjunction \[\varphi(x) = \bigvee\limits_{n \in \omega}
  \tau^{-n}(x) \in \sA\] and that the set $\sA$ is defined by a
  quantifier-free continuous first-order formula.  The isomorphism $f$
  must preserve satisfaction of $\varphi$ --- that is, $\sI \models
  \varphi(x)$ if and only if $\sJ \models \varphi(f(x))$.
  Now let $M$ be a probabilistic Turing machine such that $M(x,s)$ is
  the minimum value of $\bigwedge\limits_{n \leq k}\tau^{-n}(x) \in
  \sA$, where $k$ ranges over all numbers less than or equal to $s$.
  Then $M$ witnesses that $f(\fA)$ is probabilistically computably
  enumerable.
\end{proof}

\section{Time Complexity of Structures}\label{timecomp}

One of the most important applications of probabilistic Turing
machines is their role in computational complexity theory (see
\cite{kozenbk, papadimitriou, impwig}).  Let $P$ be some decision problem.
We say that $Q$ is of class $\mathsf{RP}$ if and only if there is a
probabilistic Turing machine $M_Q$, halting in time polynomial in the
length of the input, such that if $x \in Q$, then $M_Q$ accepts $x$ with
probability at least $\frac{3}{4}$, and if $x \notin Q$, then $M_Q$
rejects $x$ with probability $1$.  This class has the property that
$\mathsf{P} \subseteq \mathsf{RP} \subseteq \mathsf{NP}$. \footnote{In
  both this and the succeeding paragraph, the particular fraction
  $\frac{3}{4}$ is not critical.  Using a so-called ``Amplification
  Lemma,'' any fraction above and bounded away from $\frac{1}{2}$ will
  do \cite{sipserbk}.} 

Another complexity class of interest is the class $\mathsf{BPP}$.  We
say that $Q$ is of class $\mathsf{BPP}$ if and only if there is a
probabilistic Turing machine $M_Q$, halting in time polynomial in the
length of the input, such that if $x \in Q$, then $M_Q$ accepts $x$
with probability at least $\frac{3}{4}$, and if $x \notin Q$, then
$M_Q$ rejects $x$ with probability at least $\frac{3}{4}$.  Here we
know that $\mathsf{RP} \subseteq \mathsf{BPP} \subseteq \Sigma_2^p
\cap \Pi_2^p$.

\begin{dfn}[Cenzer--Remmel \cite{cenzerremmelptime}]\label{cptime} Let $\sA$ be a
  computable structure.  We say that $\sA$ is uniformly polynomial
  time if the atomic diagram of $\sA$ is a polynomial time
  set.\end{dfn}

Clearly, $A \in \mathsf{P}$ if and only if the structure $(\omega, A)$ is
polynomial time.  Also, for any polynomial time structure $\fM$ and
any quantifier-free definable $A \subseteq \fM^n$, we have $A \in
\mathsf{P}$.  We can extend Definition \ref{cptime} in a routine way
for probabilistically computable structures.

\begin{dfn} We say that a probabilistically computable structure is
  polynomial time if and only if there is some probabilistic Turing
  machine $T$ such that, for every pair $(\varphi, p) \in D(\fM)$ the
  machine $T$ halts in polynomial time and accepts $\varphi$ with
  probability $p$.\end{dfn}

Now we can characterize the members of $\mathsf{BPP}$ in terms of
continuous weak structures.

\begin{thm} The class $\mathsf{BPP}$ can be identified with the class
  of quantifier-free definable sets in polynomial time
  probabilistically computable structures in the following way:
\begin{enumerate}
\item Let $A \in \mathsf{BPP}$ be a subset of $\omega$.  Then there is
  a polynomial time probabilistically computable weak structure $\fM$ and a polynomial
  time computable function $f: \omega \to \fM$ such that there is a
  quantifier-free formula $\varphi(x)$ such that $\varphi(x) \leq
  \frac{1}{4}$ for $x \in f(A)$ and $\varphi(x) \geq \frac{3}{4}$ for
  $x \in \fM - f(A)$.
\item Let $\fM$ be a polynomial time probabilistically
  computable weak structure, and let $A, B$ be quantifier-free disjoint
  definable subsets of $\fM^n$, where \[\inf{ \{d(x,y) : x \in A, y \in
  B\}} > 0\] and $A \cup B$ is classically computably enumerable.  Then
  $A$ and $B$ are each of class $\mathsf{BPP}$.
\end{enumerate}
\end{thm}

\begin{proof} Toward the first point, let $M$ be a probabilistic
  Turing machine witnessing that $A \in \mathsf{BPP}$.  We let $\fM$
  be the structure $(\omega, \sA)$, where $\sA$ is a unary predicate
  and $\sA(x)$ is the probability that $M$ accepts $x$.  We give $\fM$
  the discrete metric.

For the second point, let $A$ be defined by $\varphi(\bar{x})$, and $B$ by
  $\psi(\bar{x})$.  Now for $\bar{a} \in A \cup B$, to check whether
  $\bar{a} \in A$, we compute $\fM(\varphi(\bar{a}) \dotminus
  \psi(\bar{a}))$.  The computation runs in polynomial time, and
  $\bar{a}$ is accepted with probability at least $\frac{1}{2} + \inf{\{d(\bar{a},
  y) : y \in B\}}$ when $\bar{a} \in A$ and
  with probability at most $\frac{1}{2} - \inf{\{d(\bar{a}, y) : y \in
  B\}}$ when $\bar{a} \in B$.\end{proof}

\bibliographystyle{amsplain}
\bibliography{pnp}

\end{document}